\documentclass[12pt, a4paper]{article}
\usepackage[utf8]{inputenc}

\usepackage[margin = 1in]{geometry}
\usepackage{amsfonts,amsmath,amssymb,amsthm}
\usepackage{graphicx,subcaption}
\usepackage{booktabs}
\usepackage{lipsum}
\usepackage{hyperref}
\usepackage{multicol}
\hypersetup{
   colorlinks = true,
   citecolor = red,
   urlcolor = purple,
   linkcolor = blue
}

\usepackage{mathrsfs}
\usepackage{multirow}
\usepackage[makeroom]{cancel}
\usepackage{tikz}
\usepackage{framed,color}
\definecolor{shadecolor}{rgb}{1,0.8,0.3}
\usepackage{fancybox}
\usepackage{caption}
\usepackage{algorithm,algorithmic} 
\usepackage{multicol}
\usepackage{aliascnt}
\usepackage{setspace}
\usepackage{longtable}

\usetikzlibrary{decorations.markings,arrows}

\allowdisplaybreaks

\title{\textbf{On Asymptotic Behaviors of Stepwise Multiple Testing Procedures}}

\date{}

\author{Monitirtha Dey \\
        Indian Statistical Institute, Kolkata \\
        \textit{\href{monitirtha.d\_r@isical.ac.in}{monitirtha.d\_r@isical.ac.in}, 
        \href{monitirthadey3@gmail.com}{monitirthadey3@gmail.com}}}

\begin{document}
\bibliographystyle{plainnat}

\maketitle
\theoremstyle{plain}
\newtheorem{axiom}{Axiom}
\newtheorem{remark}{Remark}
\newtheorem{corollary}{Corollary}[section]
\newtheorem{claim}[axiom]{Claim}
\newtheorem{theorem}{Theorem}[section]
\newtheorem{lemma}{Lemma}[section]

\newaliascnt{lemmaa}{theorem}
\newtheorem{lemmaa}[lemmaa]{Theorem}
\aliascntresetthe{lemmaa}
\providecommand*{\lemmaaautorefname}{Theorem}



\theoremstyle{plain}
\newtheorem{exa}{Example}
\newtheorem{rem}{Remark}
\newtheorem{proposition}{Proposition}

\theoremstyle{definition}
\newtheorem{definition}{Definition}
\newtheorem{example}{Example}







\begin{abstract}
Stepwise multiple testing procedures have attracted several statisticians for decades and are also quite popular with statistics users because of their technical simplicity. The Bonferroni procedure has been one of the earliest and most prominent testing rules for controlling the familywise error rate (FWER). A recent article established that the FWER for the Bonferroni method asymptotically (i.e., when the number of hypotheses becomes arbitrarily large) approaches zero under any positively equicorrelated multivariate normal framework. However, similar results for the limiting behaviors of FWER of general stepwise procedures are nonexistent. The present work addresses this gap in a unified manner by studying the limiting behaviors of the FWER of several stepwise testing rules for correlated normal setups. Specifically, we show that the limiting FWER approaches zero for any step-down rule (e.g., Holm's method) provided the infimum of the correlations is strictly positive.
We also establish similar limiting zero results on FWER of other popular multiple testing rules, e.g., Hochberg's and Hommel's procedures. We then extend these results to any configuration of true and false null hypotheses. It turns out that, within our chosen asymptotic framework, the Benjamini-Hochberg method can hold the FWER at a strictly positive level asymptotically under the equicorrelated normality. We finally discuss the limiting powers of various procedures.
\end{abstract}


\vspace{2mm}

\section{Introduction\label{sec:1}}

 
Large-scale multiple testing problems arising in various scientific disciplines often study correlated variables simultaneously. For example, in microRNA expression data, several genes may cluster into groups through
their transcription processes and possess high correlations. The data observed from different locations and time periods in public health studies are generally spatially or serially correlated. fMRI studies and multistage clinical trials also involve variables with complex and unknown dependencies. Consequently, the study of the effect of correlation on dependent test statistics in simultaneous inference problems has attracted considerable attention recently. 

Benjamini and Yekutieli \cite{BY} proved that the Benjamini-Hochberg procedure \cite{BH} controls the false discovery rate (FDR) at the desired level under positive regression dependency. Sarkar \cite{Sarkar2002} established some general results on FDR control under dependence. Storey and Tibshirani \cite{StoreyTib} proposed methodologies for estimating the FDR for dependent test statistics. Simultaneous testing methods under dependence have also been studied by Sun and Cai~\cite{suncai}, Efron~\cite{efron2007}, Liu, Zhang and Page~\cite{LZP} among others. Efron~\cite{efron2010} mentions that the correlation penalty on the summary statistics depends on the root mean square (RMS) of the correlations. Efron~\cite{Efron2010book} contains an excellent review of the relevant literature. Finner and Roters \cite{Finner2001a} discussed the behavior of expected type I errors of multiple level-$\alpha$ single-step test procedures based on exchangeable test statistics. They also studied \cite{Finner2001b} asymptotic (i.e., when the number of hypotheses tends to infinity) properties of the supremum of the expected type I error rate (EER) for some FDR-controlling stepwise procedures under independence. Fan et al. \cite{Fan2012} proposed a method of dealing with correlated test statistics with a known covariance structure. They capture the association between correlated statistics using the principal eigenvalues of the covariance matrix. Fan and Han \cite{Fan2017} extended this work when the underlying dependence structure is unknown. Qiu et al. \cite{Qiu} demonstrated that many FDR controlling procedures lose power significantly under dependence. Huang and Hsu \cite{HuangHsu} remark that stepwise decision rules based on modeling of the dependence structure are in general superior to their counterparts that do not consider the correlation.

There is relatively little literature on the performance of FWER controlling procedures under dependence. Das and Bhandari ~\cite{dasbhandari} have established that the Bonferroni FWER is asymptotically a convex function in correlation $\rho$ under the equicorrelated normal framework. Consequently, they show that the Bonferroni FWER is bounded above by $\alpha(1-\rho)$, $\alpha$ being the target level. In their recent article \cite{deybhandari}, Dey and Bhandari have improved this result by showing that the Bonferroni FWER asymptotically goes to zero for any strictly positive $\rho$. They have also extended this to arbitrarily correlated setups where the limiting infimum of the correlations is strictly positive. Dey \cite{deycstm} has obtained upper bounds on the Bonferroni FWER in the equicorrelated and general setups with small and moderate dimensions. Finner and Roters \cite{Finner2002} derived explicit formulas for the distribution of the number of falsely rejected hypotheses in single-step, step-down and step-up methods under the assumption of independent $p$-values. However, the role of correlation on the limiting behavior of the FWER for stepwise procedures is much less explored. 

The present work addresses this problem by theoretically investigating the limiting FWER values of general step-down procedures under the correlated normal setup. These results provide new insights into the behavior of step-down decision procedures. By establishing the limiting performances of commonly used step-up methods, e.g.,  the Benjamini-Hochberg method and the Hochberg method, we have elucidated that the class of step-up procedures does not possess a similar \textit{universal asymptotic zero} result as obtained in the case of step-down procedures. It is also noteworthy that most of our results are quite general since they accommodate any combination of true and false null hypotheses. We have also obtained the limiting powers of the stepwise procedures.

This paper is structured as follows. We first formally introduce the framework with relevant notations and summarize some results on the limiting behavior of the Bonferroni procedure in the next section. Section 3 studies in detail the limiting behaviors of the FWER of step-down procedures in equicorrelated and general normal setups. Section 4 is dedicated to similar results on Hochberg's and Benjamini-Hochberg procedures. Hommel's stepwise procedure is studied in Section 5 while Section 6 illustrates the limiting powers of the stepwise procedures. We outline our contributions and discuss related problems briefly in Section 7.

\section{Preliminaries} 
\subsection{Testing Framework}
Here we discuss the simultaneous inference problem through a \textit{Gaussian sequence model} framework\cite{dasbhandari, deybhandari, deycstm, Finner2001a, FDR2007}:
$$X_{i} {\sim} \mathcal{N}(\mu_{i},1),  \quad i \in \{1, \ldots,n\}$$
where $X_{i}$’s are dependent. The variances are taken to be unity since the literature on the asymptotic multiple testing theory often assumes that the variances are known (see, e.g., \cite{Abramovich, Bogdan, dasbhandari, deybhandari, Donoho}). We are interested in the following multiple testing problem: 
$$H_{0i}:\mu_{i}=0 \quad vs \quad H_{1i}:\mu_{i}>0, \quad 1 \leq i \leq n.$$ The intersection null hypothesis (also called the global null) $ H_{0}=\bigcap_{i=1}^{n} H_{0 i}$ states that each $\mu_{i}$ is zero. Let $\mathcal{A}$ denote the set of indices from $\{1,\ldots,n\}$ for which $H_{0i}$ is true. So, under the global null, $\mathcal{A}$ is $\{1, \ldots, n\}$. Throughout this work,  $\Phi(\cdot)$ denotes the cumulative distribution function of $N(0,1)$ distribution and $\alpha \in (0,1)$ denotes the target level of FWER control.

Let $V_{n}(T)$ and $R_n(T)$ respectively denote the number of type I errors and the number of rejected hypotheses of a multiple testing procedure (MTP henceforth) $T$ and $\alpha$ be the desired level of FWER control. The FWER of procedure $T$ is given by
\begin{equation}
    FWER_{T}(n, \alpha, \Sigma_n)= \mathbb{P}_{\Sigma_{n}}(V_{n}(T)\geq 1)
    \label{defFWER}
\end{equation}
where $\Sigma_{n}$ is the covariance matrix of $(X_1, \ldots, X_n)$. 
It is noteworthy that FWER is \textit{not} the probability of making any type I errors when the global null hypothesis $H_{0}$ is true. A MTP is said to have \textit{weak control} of the FWER if the FWER is less than or equal to the test level under the global null hypothesis. It has  \textit{strong control} of the FWER if the FWER is
less than or equal to the level of the test under any configuration of true and false null hypotheses. In many of our results, we shall consider the probability in the r.h.s of \eqref{defFWER} under the intersection null $H_{0}$ at first (and take that as the definition of FWER) for the sake of technical simplicity. Then we shall extend the results obtained in this case to any combination of true and false null hypotheses.

The present work studies the limiting behaviors of $FWER_T$ for $T$ belonging to a broad class of MTPs under two dependent setups:
\begin{enumerate}
\item The equicorrelated setup:
$$\operatorname{Corr}\left(X_i, X_j\right)=\rho \quad \forall i \neq j \quad(\rho \geq 0).$$
\item The arbitrarily correlated setup:
$$\operatorname{Corr}\left(X_i, X_j\right)=\rho_{ij} \quad \forall i \neq j \quad(\rho_{ij} \geq 0).$$
\end{enumerate}
The equicorrelated setup \cite{Cohen2009, dasbhandari, deybhandari, deycstm, Finner2001a, FDR2007} is the intraclass covariance matrix model, characterizing the exchangeable situation. Although this is a special case of the second one, we are considering them separately since the proof of the result in the general case is based on the corresponding results in the equicorrelated case. 
 The equicorrelated setup also encompasses the problem of comparing a control against several treatments. However, many scientific disciplines involve variables with more  complex dependence structure (e.g., fMRI studies). These complex dependence scenarios need to be tackled with more general covariance matrices \cite{deybhandari, deycstm}. The arbitrarily correlated setup also includes the successive correlation covariance matrix, which covers change point problems \cite{Cohen2009}. 

Throughout this work, $M_{n}(\rho)$ denotes the $n \times n$ matrix with each diagonal entry equal to 1 and each off-diagonal entry equal to $\rho$. Also, $\Sigma_{n}$ denotes the $n \times n$ correlation matrix with $(i,j)$'th entry equal to $\rho_{ij}$, $i \neq j$. 


\subsection{The Bonferroni Procedure}
Let $\mathcal{A}$ denote the set of indices from $\{1,\ldots,n\}$ for which $H_{0i}$ is true.
The classic Bonferroni procedure \cite{Bonferroni} is the best-known and one of the most frequently used MTP for controlling FWER. This single-step method sets the same cut-off for all the hypotheses. In one-sided settings, it rejects $H_{0i}$ if $X_{i}>\Phi ^{-1}(1-\alpha/n) (=c_{\alpha,n}, \text{say})$. So the Bonferroni FWER (for the covariance matrix $\Sigma_n$) is defined by
\begin{equation*}
    FWER_{Bon}(n, \alpha, \Sigma_n) =\mathbb{P}_{\Sigma_n}\left(X_{i}>c_{\alpha,n}\right. \text{for some} \left. i \in \mathcal{A}\right)
=\mathbb{P}_{\Sigma_n}\bigg(\bigcup_{i \in \mathcal{A}}\{X_i > c_{\alpha,n}\}\bigg). 
\end{equation*}
We write $FWER_{T}(n,\alpha, M_n(\rho))$ as $FWER_{T}(n,\alpha, \rho)$ for simpler notation. Das and Bhandari~\cite{dasbhandari} obtain the following in the equicorrelated case:
\begin{theorem}\label{thm2.1}
Given any $\alpha \in (0,1)$ and $\rho \in [0,1]$, $FWER_{Bon}(n, \alpha, \rho)$ is asymptotically bounded by $\alpha(1-\rho)$ under the global null hypothesis.
\end{theorem}
Dey and Bhandari \cite{deybhandari} improve this result as follows: 
\begin{theorem}\label{thm2.2}
Given any $\alpha \in (0,1)$ and $\rho \in (0,1]$, we have
$$\lim_{n \to \infty} FWER_{Bon}(n, \alpha, \rho) = 0$$
under any configuration of true and false null hypotheses.
\end{theorem}

The proofs of Theorem \ref{thm2.1} and \ref{thm2.2} exploit an well known result on equicorrelated multivariate normal variables with equal marginal variances. Under the global null hypothesis, the sequence $\left\{X_{r}\right\}_{r \geq 1}$ is exchangeable in the equicorrelated normal set-up. In other words,  $$\left(X_{i_{1}}, \ldots, X_{i_{k}}\right) \sim N_{k}\left(\mathbf{0}_{\mathbf{k}},(1-\rho) I_{k}+\rho J_{k}\right))$$ 
where $J_{k}$ is the $k \times k$ matrix of all ones. Thus, for each $i \geq 1$, $X_{i}=\theta+Z_{i}$
where $\theta$ has a normal distribution with mean $0$, independent of $\left\{Z_{n}\right\}_{n \geq 1}$ and $Z_{i}$’s are i.i.d normal random variables. $\operatorname{Cov}\left(X_{i}, X_{j}\right)=\rho$ gives $\operatorname{Var}(\theta)=\rho$. Hence, 
$\theta \sim \mathcal{N}(0, \rho)$ and $Z_{i} \overset{iid}{\sim} \mathcal{N}(0,1-\rho)$ for each $i \geq 1$. \\

The authors in \cite{deybhandari} also extend Theorem \ref{thm2.2} to arbitrarily correlated normal setups:
\begin{theorem}\label{thm2.3}
Let $\Sigma_n$ be the correlation matrix of $X_1, \ldots, X_n$ with $(i,j)$’th entry $\rho_{ij}$ such that $\liminf \rho_{ij}=\delta>0$. Then, for any $\alpha \in (0,1)$, we have
$$\lim_{n \to \infty}FWER_{Bon}(n,\alpha,\mathbf{\Sigma}_{n}) = 0$$
under any configuration of true and false null hypotheses.
\end{theorem}
Theorem \ref{thm2.3}, a much stronger result than Theorem \ref{thm2.1}, highlights the fundamental problem of using Bonferroni method in a
simultaneous testing problem. The authors in \cite{deybhandari} establish Theorem \ref{thm2.3} using a famous inequality due to Slepian ~\cite{Slepian}: 
\begin{theorem}\label{thm2.4}
Let $\mathbf{X}$ follow $\mathbf{N}_{k}(\mathbf{0}, \mathbf{\Sigma})$, where $\mathbf{\Sigma}$ is a $k \times k$ correlation matrix. Let $\mathbf{a}=\left(a_{1}, \ldots, a_{k}\right)^{\prime}$ be an arbitrary but fixed real 
vector. Consider the quadrant probability
$$g(k, \mathbf{a}, \mathbf{\Sigma})=\mathbb{P}_{\mathbf{\Sigma}}\left[\bigcap_{i=1}^{k}\left\{X_{i} \leqslant a_{i}\right\}\right].$$ Let $\mathbf{R}=\left(\rho_{i j}\right)$ and $\mathbf{T}=\left(\tau_{i j}\right)$ be two positive semidefinite correlation matrices. If $\rho_{i j} \geqslant \tau_{i j}$ holds for all $i, j$, then $g(k, \mathbf{a}, \mathbf{R}) \geq g(k, \mathbf{a}, \mathbf{T})$, i.e
$$
\mathbb{P}_{\mathbf{\Sigma}=\mathbf{R}}\left[\bigcap_{i=1}^{k}\left\{X_{i} \leqslant a_{i}\right\}\right] \geqslant \mathbb{P}_{\mathbf{\Sigma}=\mathbf{T}}\left[\bigcap_{i=1}^{k}\left\{X_{i} \leqslant a_{i}\right\}\right]
$$
holds for all $\mathbf{a}=\left(a_{1}, \ldots, a_{k}\right)^{\prime} .$ Moreover, the inequality is strict if $\mathbf{R}, \mathbf{T}$ are positive definite and if the strict inequality $\rho_{i j}>\tau_{i j}$ holds for some $i, j$.
\end{theorem}


Throughout this work, $P_{i}$ denotes the $p$-value corresponding to the $i$-th null hypothesis $H_{0i}$, $1 \leq i \leq n$. Also, let $P_{(1)} \leqslant \ldots \leqslant P_{(n)}$ be the ordered $p$-values. Let the null hypothesis corresponding to the p-value $P_{(i)}$ be denoted as $H_{(0i)}$, $1 \leq i \leq n$. 



\subsection{Step-down and Step-up Procedures} 

Single-step MTPs (e.g., Bonferroni's method, Sidak's method) compare the individual test statistics to the corresponding cut-offs simultaneously, and they stop after performing this simultaneous ‘joint’ comparison. Often stepwise methods possess greater power than the single-step procedures, while still controlling FWER (or, in general, the error rate under consideration) at the desired level.

Consider the simplex 
$$\mathcal{S}_n=\left\{\mathbf{t}=\left(t_1, \ldots, t_n\right) \in \mathbb{R}^n: 0 \leq t_1 \leq \ldots \leq t_n \leq 1\right\}.$$

A $p$-value based step-down MTP uses a vector of cutoffs $\textbf{u} =\left(u_1, \ldots, u_n\right) \in \mathcal{S}_n$, and works as follows. The step-down method rejects a hypothesis $H_{(i)}$ if and only if $P_{(j)} \leq u_{j}$ for all $j \leq i$. In other words, the step-down MTP compares the most significant $p$-value $P_{(1)}$ with the smallest $u$-value $u_{1}$ at first and so on. One can also formally describe a step-down MTP as follows. Let $m_1=\max \left\{i: P_{(j)} \leq u_j\right.$ for all $\left.j=1, \ldots, i\right\}$. Then the step-down procedure based on critical values $\textbf{u}$ rejects $H_{(1)}, \ldots, H_{\left(m_1\right)}$. 

\begin{example} 
The Bonferroni method is a step-down procedure with $u_i=\alpha / n$, $i=1, \ldots, n$.
\end{example}

\begin{example} 
The Sidak method is a step-down MTP with $u_i=1 - (1-\alpha)^{1 / n}$, $i=1, \ldots, n$.
\end{example}

\begin{example} 
The Holm \cite{Holm} method is a popular step-down MTP with $u_i=\alpha /(n-i+1), i=1, \ldots, n$.
\end{example}

\begin{example} 
Benjamini and Liu \cite{BL} introduced a step-down MTP with
$$u_i=\min \left(1, \frac{n q}{(n-i+1)^2}\right), \quad 1 \leq i \leq n \quad(0<q<1).$$
\end{example}

\begin{example} 
Benjamini and Liu studied another step-down MTP in \cite{BL2} with
$$
u_i=1 - \Bigg[1 - \min \left(1, \frac{n q}{n-i+1}\right)\Bigg]^{1/{(n-i+1)}}, \quad 1 \leq i \leq n \quad(0<q<1).
$$
\end{example}

\begin{example} 
Benjamini and Liu mentioned in \cite{BL2} a Holm-type procedure with the critical values
$$
u_i=1 - (1 - q)^{1/{(n-i+1)}}, \quad 1 \leq i \leq n \quad(0<q<1).
$$
\end{example}

The step-up MTP also utilizes set of critical values, say $\textbf{u} =\left(u_1, \ldots, u_n\right) \in \mathcal{S}_n$. But the step-up method is inherently different from the step-down method in the sense that it starts by comparing the least significant $p$-value $P_{(n)}$ with the largest $u$-value $u_n$ and so on. Formally, the step-up method based on critical values $\textbf{u}$ rejects the hypotheses $H_{(1)}, \ldots, H_{\left(m_2\right)}$, where $m_2=\max \left\{i: P_{(i)} \leq u_{i}\right\}$. If such a $m_2$ does not exist, then the procedure does not reject any null hypothesis.

\begin{example} 
The Bonferroni correction is also a step-up MTP, where $u_i=\alpha / n, i=1, \ldots, n$.
\end{example}

\begin{example} 
The Sidak method is also a step-up procedure with $u_i=1 - (1-\alpha)^{1 / n}$, $i=1, \ldots, n$.
\end{example}

\begin{example} 
The Hochberg \cite{Hochberg} method is a popular step-up MTP with $u_i=\alpha /(n-i+1)$.
\end{example}

\begin{example} 
The classic Benjamini-Hochberg \cite{BH} method is a step-up procedure with $u_i=i \alpha / n$.
\end{example}

\section{Limiting FWER of Step-down Procedures}
Holm method is a step-down MTP which uses modified critical values and utilizes the Bonferroni inequality. It controls the FWER under any dependence of the test statistics. The authors in \cite{deybhandari} show the following result on the limiting FWER of Holm's method \cite{Holm} under the equicorrelated normal framework:

\begin{theorem}\label{thm3.1}
Suppose $\mu^{\star} = \sup \mu_{i} < \infty$. Then, under any configuration of true and false null hypotheses, we have
$$\displaystyle \lim_{n \to \infty} FWER_{Holm}(n, \alpha, \rho) = 0 \quad \text{for all} \hspace{2mm} \alpha \in (0,1) \hspace{2mm} \text{and} \hspace{2mm} \rho \in (0,1].$$
\end{theorem}
We extend this result to arbitrary correlated normal setups:

\begin{theorem}\label{thm3.2}
Let $\Sigma_n$ be the correlation matrix of $X_1, \ldots, X_n$ with $(i,j)$’th entry $\rho_{ij}$ such that $\liminf \rho_{ij}=\delta>0$. Suppose $\mu^{\star} = \sup \mu_{i} < \infty$. Then, for any $\alpha \in (0,1)$,
$$\lim_{n \to \infty}FWER_{Holm}(n,\alpha,\mathbf{\Sigma}_{n}) = 0$$
under any configuration of true and false null hypotheses.
\end{theorem}

\begin{proof}[\textbf{\upshape Proof of Theorem \ref{thm3.2}.}]
Let $R_{n}(H)$ and $V_{n}(H)$ denote the number of rejected hypotheses and type I errors of Holm’s MTP respectively. Then,

$$\begin{aligned}
FWER_{Holm}(n,\alpha,\mathbf{\Sigma}_{n})
= & \mathbb{P}_{\Sigma_{n}}(V_{n}(H)\geq 1)\\
\leq &\mathbb{P}_{\Sigma_{n}}(R_{n}(H)\geq 1)\\
= & \mathbb{P}_{\Sigma_{n}}(P_{(1)} \leq \alpha/n)\\
= & \mathbb{P}_{\Sigma_{n}}(X_{(n)} \geq c_{\alpha,n}).\\
\end{aligned}$$
Theorem \ref{thm2.4} gives $\mathbb{P}_{\Sigma_{n}}(X_{(n)} \geq c_{\alpha,n}) \leq \mathbb{P}_{M_{n}(\rho)}(X_{(n)} \geq c_{\alpha,n})$. Without any loss of generality, we may assume $X_{i} \sim N(\mu_{i},1)$ ($\mu_i >0$) for $1 \leq i \leq n_{1}$ and for $n_{1} < i \leq n$, $X_{i} \sim N(0,1)$. 

\noindent Thus, 
$$\begin{aligned}
 FWER_{Holm}(n,\alpha,\mathbf{\Sigma}_{n}) 
\leq & \mathbb{P}_{M_{n}(\rho)}(X_{(n)} \geq c_{\alpha,n}) \\
= & 1 - \mathbb{P}_{M_{n}(\rho)}(X_{(n)} \leq c_{\alpha,n})\\
=&1 - \mathbb{P}_{M_{n}(\rho)}\left(X_{i} \leqslant c_{\alpha, n} \quad \forall i=1,2, \ldots, n\right)\\
=&1 - \mathbb{P}_{M_{n}(\rho)}\Bigg[\bigcap_{i=1}^{n_1} \{\theta+Z_{i}+\mu_{i} \leqslant c_{\alpha, n} \} \bigcap \bigcap_{i=n_1 +1}^{n} \{\theta+Z_{i} \leqslant c_{\alpha, n} \} \Bigg]\\
=&1 - \mathbb{E}_{\theta}\Bigg[\bigg\{\prod_{i=1}^{n_1} \Phi \left(\frac{c_{\alpha, n}-\theta - \mu_{i}}{\sqrt{1-\rho}}\right)\bigg\}\cdot \Phi^{n-n_1}\left(\frac{c_{\alpha, n}-\theta}{\sqrt{1-\rho}}\right)\Bigg]\\
\leq &1 - \mathbb{E}_{\theta} \left[\Phi^{n}\left(\frac{c_{\alpha, n}-\theta - \mu^{\star}}{\sqrt{1-\rho}}\right)\right].
\end{aligned}$$
The last quantity above tends to zero asymptotically since $\mu^{\star} < \infty$.\end{proof}

\noindent The proof of Theorem \ref{thm3.2} also gives us the following useful result:
\begin{corollary}\label{cor1}
Let $\Sigma_n$ be the correlation matrix of $X_1, \ldots, X_n$ with $(i,j)$’th entry $\rho_{ij}$ such that $\liminf \rho_{ij}=\delta>0$. Suppose $\mu^{\star} = \sup \mu_{i} < \infty$. Then, for any $\alpha \in (0,1)$,
$$\lim_{n \to \infty}\mathbb{P}_{\Sigma_{n}}\bigg(R_{n}(H) \geq 1\bigg) = \lim_{n \to \infty}\mathbb{P}_{\Sigma_{n}}\bigg(P_{(1)} \leq \alpha/n\bigg) = 0$$
under any configuration of true and false null hypotheses.
\end{corollary}

We note that Corollary \ref{cor1} is infact a stronger result than Theorem \ref{thm3.2} since Corollary \ref{cor1} concerns the probability of rejecting any null while Theorem \ref{thm3.2} considers with the false rejections only. We extend Theorem \ref{thm3.2} to any step-down MTP below:

\begin{theorem}\label{thm3.3}
Let $\Sigma_n$ be the correlation matrix of $X_1, \ldots, X_n$ with $(i,j)$’th entry $\rho_{ij}$ such that $\liminf \rho_{ij}=\delta>0$. Suppose $\mu^{\star} = \sup \mu_{i} < \infty$ and $T$ is any step-down MTP controlling FWER at level $\alpha \in (0,1)$. Then, for any $\alpha \in (0,1)$,
$$\lim_{n \to \infty}FWER_{T}(n,\alpha,\mathbf{\Sigma}_{n}) = 0$$
under any configuration of true and false null hypotheses.
\end{theorem}
Theorem \ref{thm3.3} can be established using the following result due to Gordon and Salzman \cite{Gordon}.

\begin{theorem}\label{thm3.4}
Let $T$ be a step-down MTP based on the set of cut-offs $\mathbf{u} \in \operatorname{Simp}^n$. If $FWER_{T} \leq \alpha<1$, then $u_i \leq \alpha /(n-i+1), i=1, \ldots, n$.
\end{theorem}
\begin{proof}[\textbf{\upshape Proof of Theorem \ref{thm3.3}.}] We have,
$$\begin{aligned}
FWER_{T}(n,\alpha,\mathbf{\Sigma}_{n})
=\mathbb{P}_{\Sigma_{n}}(V_{n}(T)\geq 1)
\leq &\mathbb{P}_{\Sigma_{n}}(R_{n}(T)\geq 1)\\
= & \mathbb{P}_{\Sigma_{n}}(P_{(1)} \leq u_{1})\\
\leq & \mathbb{P}_{\Sigma_{n}}(P_{(1)} \leq \alpha/n). 
\end{aligned}$$
The last step above follows since we have $u_1 \leq \alpha/n$ from Theorem \ref{thm3.4}. The rest is obvious from Corollary \ref{cor1}.
\end{proof}
Theorem \ref{thm3.3} can be viewed as a \textit{universal asymptotic zero} result since it encompasses all step-down FWER controlling procedures and also accommodates any configuration of true and false null hypotheses. We also have the following interesting extension of Corollary \ref{cor1}:
\begin{corollary}\label{cor2}
Let $\Sigma_n$ be the correlation matrix of $X_1, \ldots, X_n$ with $(i,j)$’th entry $\rho_{ij}$ such that $\liminf \rho_{ij}=\delta>0$. Suppose $\mu^{\star} = \sup \mu_{i}$ is finite and $T$ is any step-down MTP  controlling FWER at level $\alpha \in (0,1)$. Then, for any $\alpha \in (0,1)$,
$$\lim_{n \to \infty}\mathbb{P}_{\Sigma_{n}}\bigg(R_{n}(T) \geq 1\bigg) = 0$$
under any configuration of true and false null hypotheses.
\end{corollary}



\section{Limiting FWER of Some Step-up Procedures}

Let us consider a step-down procedure $T_{1}$ and a step-up procedure $T_{2}$ having the identical vector of cutoffs $\mathbf{u} = (u_1, \ldots, u_n) \in \mathcal{S}_n$. We always have $m_1(T_{1}) \leq m_2(T_{2})$ where 
$m_1(T_{1})=\max \left\{i: P_{(j)} \leq u_j\right.$ for all $\left.j=1, \ldots, i\right\}$ and $m_2(T_{2})=\max \left\{i: P_{(i)} \leq u_{i}\right\}$. This implies that the step-up MTP is at least as rejective as the step-down MTP (which uses the same cutoffs). This observation steers that we might not get a similar \textit{universal asymptotic zero} result for the class of step-up MTPs as obtained in the case of step-down procedures (Theorem \ref{thm3.3}). This is indeed the case as we shall show in the next two subsections the following:
\begin{enumerate}
\item Under the equicorrelated Gaussian sequence model, the FWER of Hochberg procedure \cite{Hochberg} asymptotically approaches zero as the number of tests becomes arbitrarily large.
\item Under the equicorrelated Gaussian sequence model and under $H_0$, the Benjamini-Hochberg procedure \cite{BH} with a pre-specified FDR level controls FDR at some strictly
positive quantity which is a function of the chosen
FDR level and the common correlation, even when the number of tests approaches infinity.
\end{enumerate}
We have considered Hochberg’s MTP in particular because it uses the same vector of cutoffs as Holm's MTP (note that Holm's MTP has the ‘optimal’ critical values in the class of step-down procedures). Benjamini-Hochberg method, on the other hand, has been one of the most eminent MTPs proposed in the literature and also possesses some optimality properties both in frequentist and Bayesian paradigms of simultaneous inference (see \cite{Bogdan, Guo}).  
\subsection{Hochberg's Procedure}

Hochberg's \cite{Hochberg} MTP and Holm's sequentially rejective procedure use the same set of cutoffs; and hence, as mentioned earlier, Hochberg's method is sharper than Holm's MTP. Holm's MTP rejects a hypothesis only if its $p$-value and each of the smaller p-values are less than their corresponding cutoffs. Hochberg's method rejects all hypotheses with smaller or equal $p$-values to that of any one found less than its cutoff.

The following result depicts the limiting behavior of the FWER of Hochberg's procedure under the correlated Gaussian sequence model:

\begin{theorem}\label{thm4.1}
Consider the equicorrelated normal setup with correlation $\rho \in [0,1)$. Then, 
\begin{enumerate}
\item When $\rho=0$ (i.e., the independent normal setup), we have
$$\lim_{n \to \infty} FWER_{Hochberg}(n, \alpha, 0) \in  [1 - e^{-\alpha},\alpha]$$
under the global null hypothesis.
\item When $\rho \in (0,1)$, we have
$$\lim_{n \to \infty} FWER_{Hochberg}(n, \alpha, \rho) = 0$$
for any $\alpha \in (0,1/2)$, under the global null hypothesis.
\end{enumerate}
\end{theorem}

\begin{proof}[\textbf{\upshape Proof of Theorem \ref{thm4.1}.}]
We have, under the global null, 
\begin{align*}
FWER_{Hochberg}(n, \alpha, 0)
= & \mathbb{P}_{I_{n}}\left[\bigcup_{i=1}^n\left\{P_{(i)} \leqslant \frac{\alpha}{n-i+1}\right\}\right] \\
\geqslant & \mathbb{P}_{I_n}\left[P_{(1)} 
\leqslant \frac{\alpha}{n}\right] \\
\underset{n \rightarrow \infty}{\longrightarrow}&  1-e^{-\alpha}.
\end{align*}
Also, Hochberg's procedure controls FWER at level $\alpha$ \cite{Hochberg}.
So, $$1-e^{-\alpha} \leqslant \lim_{n \rightarrow \infty} FWER_{Hochberg }(0) \leqslant \alpha.$$
Also, $\lim_{\alpha \rightarrow 0} \frac{1-e^{-\alpha}}{\alpha}=1$. thus,  we have, as $\alpha \rightarrow 0$,
$\lim _{n \rightarrow \infty} \frac{FWER_{Hochberg}(0)}{\alpha}=1$. This completes the proof of the first part. 

\noindent When $\rho \in (0,1)$, we have,
$$
\begin{aligned}
& P_{(i)} \leqslant \frac{\alpha}{n-i+1} \\
\Longleftrightarrow & 1-\frac{\alpha}{n-i+1} \leq \Phi\left(X_{(n-i+1)}\right) \\
\Longleftrightarrow & \Phi^{-1}\left(1-\frac{\alpha}{n-i+1}\right) \leqslant U+Z_{(n-i+1)} \\
\Longleftrightarrow & \Phi^{-1}\left(1-\frac{\alpha}{n-i+1}\right) \leqslant U+\sqrt{1-\rho}\cdot \Phi^{-1}\left(1-\frac{i}{n}\right) \quad \text{(for all sufficiently large values of $n$).}
\end{aligned}
$$
Therefore, for all sufficiently large values of $n$, we have
$$
\begin{aligned}
& P_{(i)} \leqslant \frac{\alpha}{n-i+1} \\
\Longleftrightarrow & -U-\sqrt{1-\rho}\cdot \Phi^{-1}\left(1-\frac{i}{n}\right) \leqslant-\Phi^{-1}\left(1-\frac{\alpha}{n-i+1}\right) \\
\Longleftrightarrow & -U+\sqrt{1-\rho} \cdot \Phi^{-1}\left(\frac{i}{n}\right) \leqslant \Phi^{-1}\left(\frac{\alpha}{n-i+1}\right) \\
\Longleftrightarrow & \frac{-U}{\Phi^{-1}\left(\frac{\alpha}{n-i+1}\right)}+\frac{\sqrt{1-\rho} \cdot \Phi^{-1}\left(\frac{i}{n}\right)}{\Phi^{-1}\left(\frac{\alpha}{n-i+1}\right)} \geqslant 1 \quad \text{(since $\alpha \in (0,1/2))$}\\
\Longleftrightarrow & \displaystyle \lim_{n \to \infty} \frac{\Phi^{-1}\left(\frac{i}{n}\right)}{\Phi^{-1}\left(\frac{\alpha}{n-i+1}\right)} \geq \frac{1}{\sqrt{1-\rho}}.
\end{aligned}
$$
Thus, we have $i/n < 1/2$, because otherwise the limiting ratio of $\Phi^{-1}\left(\frac{i}{n}\right)$ and $\Phi^{-1}\left(\frac{\alpha}{n-i+1}\right)$ can not be positive. So, we have
$$\frac{i}{n} < \frac{\alpha}{n-i+1} <1/2.$$
This implies $i(n-i+1) < \alpha\cdot n$. But this is not valid for any value of $i$ in $\{1, \ldots, n\}$. Consequently, the limiting FWER is zero. \end{proof}

We now consider the free-combination condition \cite{Holm} under which any combination of the true and false hypotheses is possible.

\begin{theorem}\label{thm4.2}
Consider the multiple testing problem under the equicorrelated normal setup with equicorrelation $\rho\in (0,1)$. Suppose $\sup \mu_{i}$ is finite. Then, 
$$\lim_{n \to \infty} FWER_{Hochberg}(n, \alpha, \rho) = 0.$$
for any $\alpha \in (0,1/2)$.
\end{theorem}

\begin{proof}[\textbf{\upshape Proof of Theorem \ref{thm4.2}.}]
Suppose $\mathcal{A}$ denote the set of indices from $\{1,\ldots,n\}$ for which $H_{0i}$ is true. We have, 
\begin{align*}
& FWER_{Hochberg}(n, \alpha, \rho) \\
= &\mathbb{P}_{\Sigma_{n}}\bigg(V_{n}(Hochberg) \geq 1\bigg) \\
\leq &\mathbb{P}_{\Sigma_{n}}\bigg(R_{n}(Hochberg) \geq 1\bigg) \\
=  & \mathbb{P}_{M_{n}(\rho)}\left[\bigcup_{i=1}^n\left\{P_{(i)} \leqslant \frac{\alpha}{n-i+1}\right\}\right]\\
= & \mathbb{P}_{M_{n}(\rho)}\left[\bigcup_{i=1}^n\left\{\Phi^{-1}\left(1-\frac{\alpha}{n-i+1}\right) \leq U + Z_{(n-i+1)}\right\}\right] \\
= &\mathbb{P}_{M_{n}(\rho)}\left[\bigcup_{i \in \mathcal{A}}\left\{c_{\alpha, n-i+1} \leq U + Z_{(n-i+1)}\right\} \bigcup \bigcup_{i \in \mathcal{A}^{c}}\left\{c_{\alpha, n-i+1} \leq U + Z_{(n-i+1)}\right\}\right].
\end{align*}
In the last step above, $c_{\alpha, n-i+1}$ denotes $\Phi^{-1}\left(1-\frac{\alpha}{n-i+1}\right)$.

Here, $Z_{i} = \mu_{i} +V_{i}$ where $V_{i}$'s are i.i.d $N(0,1-\rho)$ variables, $\mu_{i}$ is zero for $i \in \mathcal{A}$ and is strictly positive otherwise. So, $Z_{i}$ always lies in $[V_{i}, V_{i} +\sup \mu_i]$. This implies, for all $i \in \{1, \ldots, n\}$,
$$Z_{(n-i+1)} \leq V_{(n-i+1)} +\sup \mu_i.$$
Hence, we have
$$FWER_{Hochberg}(n, \alpha, \rho) \leq \mathbb{P}_{M_{n}(\rho)}\left[\bigcup_{i=1}^n\left\{\Phi^{-1}\left(1-\frac{\alpha}{n-i+1}\right) \leq U + V_{(n-i+1)} + \sup \mu_{i}\right\}\right].$$

 Proceeding exactly in the same way as in the earlier proof, one obtains that the upper bound tends to zero asymptotically, and consequently, the limiting FWER is zero.  \end{proof}

 Note that, we also have the following result regarding the number of rejections:

 \begin{corollary}\label{cor3}
Consider the equicorrelated normal setup with equicorrelation $\rho\in (0,1)$. Suppose $\sup \mu_{i}$ is finite. Then, 
$$\lim_{n \to \infty} \mathbb{P}_{M_n(\rho)}\bigg(R_n(Hochberg) \geq 1\bigg)= 0.$$
for any $\alpha \in (0,1/2)$ and under any configuration of true and false null hypotheses.
\end{corollary}

\subsection{Benjamini-Hochberg Procedure}
The Benjamini-Hochberg method is the first FDR controlling procedure \cite{FDR2007}. Originally shown to be a valid FDR controlling method for independent $p$-values in \cite{BH}, it controls the FDR even if the test statistics exhibit some special dependence structure (e.g., the positive regression dependent setup). Formal treatments of these conditions
and proofs can be found in \cite{BY} and \cite{Sarkar2002}. Let $i_{\max }$ be the largest such $i$ for which $p_{(i)} \leqslant i \alpha/n$. The BH procedure rejects $H_{0(i)}$ if $i \leqslant i_{\max }$ and accepts $H_{0(i)}$ otherwise. The authors in \cite{deybhandari} evaluated the limiting FDR of Benjamini-Hochberg method. However, their proof had a technical gap. We give the revised statement along with the correct proof below:
\begin{theorem}\label{thmBH} Consider the multiple testing problem under the equicorrelated normal setup with correlation $\rho$. Then, under the global null, for all $\alpha \in (0,1)$ and for all $\rho \in (0,1)$, 
$$\lim_{n \to \infty} FDR_{BH}(n, \alpha, \rho) =  1 - \Phi\left[\inf_{t \in  (0,1)} \frac{\Phi^{-1}\left(1-t \alpha\right)-\sqrt{1-\rho} \cdot \Phi^{-1}(1-t)}{\sqrt{\rho}}\right]>0.$$
Also,
$\lim_{n \to \infty} FDR_{BH}(n, \alpha, 0) =\lim_{n \to \infty} FDR_{BH}(n, \alpha, 1) =\alpha$.
\end{theorem}

\begin{proof}[\textbf{\upshape Proof of Theorem \ref{thmBH}.}]
We have
$$F D R_{BH} =\mathbb{E}\left[\frac{V_{n}(BH)}{\max \left\{R_{n}(BH), 1\right\}}\right]=\mathbb{E}\left[\frac{V_{n}(BH)}{R_{n}(BH)} \mid V_{n}(BH)>0\right] \mathbb{P}\left(V_{n}(BH)>0\right).$$

\noindent Under the global null $H_{0}$, all $R_{n}$ rejected hypotheses are false rejections, hence 

\noindent $V_{n}(BH) / R_{n}(BH)=1$ and FDR equals FWER. We shall work with $FWER$ for the rest of this proof.

Suppose exactly $n_0$ null hypotheses are true. Then, it is a well-known fact \cite{BH, Efron2010book, Sarkar2002} that, under the independent setup,
$$FDR_{BH}(n, \alpha, \rho) = \frac{n_{0}}{n}\alpha.$$

\noindent So, under the global null, $FDR_{BH}(n, \alpha, 0) = FWER_{BH}(n, \alpha, 0) =\alpha$. Now, 

$$\begin{aligned}
p_{(i)} \leqslant \frac{i \alpha}{n} \Longleftrightarrow &  1-\Phi\left(X_{(n-i+1)}\right) \leqslant \frac{i \alpha}{n} \\
\Longleftrightarrow &  1-\frac{i \alpha}{n} \leqslant \Phi\left(X_{(n-i+1)}\right) \\
\Longleftrightarrow &  X_{(n-i+1)} \geqslant \Phi^{-1}\left(1-\frac{i \alpha}{n}\right).
\end{aligned}$$

\noindent Consequently, 
\begin{align*}
FWER_{B H}(\rho)
= & \mathbb{P}_{M_{n}(\rho)}\left[\bigcup_{i=1}^{n}\left\{P_{(i)}<\frac{i \alpha}{n}\right\}\right]
= \mathbb{P}_{M_{n}(\rho)}\left[\bigcup_{i=1}^{n}\left\{ X_{(n-i+1)} \geqslant \Phi^{-1}\left(1-\frac{i \alpha}{n}\right) \right\}\right].
\end{align*}

\noindent When $\rho=1$, $X_{i}=X_{j}$ w.p $1$. This implies
\begin{align*}
FWER_{B H}(\rho)
= \mathbb{P}\left[ \bigcup_{i=1}^{n} \left\{ X \geqslant \Phi^{-1}\left(1-\frac{i \alpha}{n}\right) \right\}\right] = & \mathbb{P}\left[ X \geqslant \Phi^{-1}\left(1-\alpha\right) \right] = \alpha \hspace{1.5mm} (\text{$X \sim N(0,1)$})
\end{align*}

\noindent Consider the case $0<\rho<1$ now. Then, $X_{i} = U_{\rho}+ Z_{i}$ where $U_{\rho} \sim N(0, \rho)$ is independent of $Z_{i} \sim N(0,1-\rho)$. Here $Z_{i}$'s are i.i.d. So, under the global null, 
$$
\begin{aligned}
FWER_{B H}(\rho)
= & \mathbb{P}_{M_{n}(\rho)}\left[\bigcup_{i=1}^{n}\left\{ X_{(n-i+1)} \geqslant \Phi^{-1}\left(1-\frac{i \alpha}{n}\right)\right\}\right]\\
= & \mathbb{P}_{M_{n}(\rho)}\left[\bigcup_{i=1}^{n}\left\{ U_{\rho}+Z_{(n-i+1)} \geqslant \Phi^{-1}\left(1-\frac{i \alpha}{n}\right)\right\}\right]\\
= & \mathbb{P}_{M_{n}(\rho)}\left[\bigcup_{i=1}^{n}\left\{U_{\rho}>\Phi^{-1}\left(1-t_{i} \alpha\right)-Z_{\left(n-n t_{i}+1\right)}\right\}\right] \quad \text{where} \hspace{2mm} t_{i}=i/n .  \\
= & \mathbb{P}_{M_{n}(\rho)}\left[\bigcup_{i=1}^{n}\left\{M>\frac{\Phi^{-1}\left(1-t_{i} \alpha\right)-Z_{\left(n-n t_{i}+1\right)}}{\sqrt{\rho}}\right\}\right],
\end{aligned}
$$
where in the last step above, $M = U_{\rho}/\sqrt{\rho} \sim N(0,1)$. For $t \in (0,1)$,
$Z_{(n-nt+1)}=Z_{\left(n\left(1-t+\frac{1}{n}\right)\right)}$ converges in probability to $(1-t)$'th quantile of the distribution of $Z_{i}$ (i.e. $\sqrt{1-\rho} \cdot \Phi^{-1}(1-t)$) as $n \rightarrow \infty$. So,
\begin{align*}
    \lim_{n \to \infty} FWER_{B H}(\rho) & = \mathbb{P}\left[\bigcup_{t \in (0,1)}\left\{M>\frac{\Phi^{-1}\left(1-t \alpha\right)-\sqrt{1-\rho} \cdot \Phi^{-1}(1-t)}{\sqrt{\rho}}\right\}\right] \\
    & = \mathbb{P}\left[\bigcup_{t \in (0,1)}\left\{M > s(t)\right\}\right] \quad \text{(say)} \\
    & = \mathbb{P}\left[M>\inf_{t \in  (0,1)} s(t)\right] \\
    & = 1 - \Phi\left[\inf_{t \in  (0,1)} s(t)\right].
\end{align*}
Now, $\displaystyle\inf_{t \in  (0,1)} s(t) \leq s(.5) <\infty$. So, $\displaystyle \Phi\left[\inf_{t \in  (0,1)} s(t)\right] <1$. Thus, $\displaystyle \lim_{n \to \infty} FWER_{B H}(\rho) >0$ for $\rho \in (0,1)$.
\end{proof}
\begin{remark}
Since we are considering the infimum of the function $s(\cdot)$ and since $s(0)=\infty=s(1)$, the previous result still holds good if one considers the closed interval $[0,1]$ in place of the open interval $(0,1)$.
\end{remark}
\begin{remark}
The authors in \cite{FDR2007} studied the (limiting) empirical distribution function of the $p$-values and used those to study limiting behaviors of FDR. Their results are derived under general distributional setups and different values of $\xi_{n}$ where $\xi_{n}$ denotes the proportion of the true nulls. Our elementary proof, in contrast, uses standard analytic tools and provides a simple closed-form expression for the limiting FDR under the global null.
\end{remark}

\subsection{Other Step-up Procedures}
We have discussed the limiting FWER values of two step-up procedures so far. In this subsection, we shall provide an upper bound on the limiting FWER of any step-up procedure satisfying some properties. Towards this, we discuss a special dependency property of test statistics introduced by Benjamini and Yekutieli \cite{BY}. They referred to this property as \textit{positive regression dependency on each one from a subset $\mathcal{A}$}, or PRDS on $\mathcal{A}$. The notion of PRDS involves increasing sets.

\begin{definition}
A set $D \subset \mathbb{R}^{k}$ is called an \textit{increasing set} if $\mathbf{a} \in D$ and $\mathbf{b} \geq \mathbf{a}$ imply that $\mathbf{b} \in D$.
\end{definition}

\begin{definition}
\textit{Property PRDS.} Let $\mathbf{X}=\left(X_1, X_2, \ldots, X_n\right)$ be the vector of test statistics. We say that the PRDS property holds on $\mathcal{A}$ if for any increasing set $D$, and for each $i \in \mathcal{A}$, $\mathbb{P}\left\{\mathbf{X} \in D \mid X_i=x\right\}$ is nondecreasing in $x$.
\end{definition}

Benjamini and Yekutieli \cite{BY} established that the Benjamini-Hochberg method controls the FDR under the PRDS property. Let $\mathbf{\alpha}_T = (\alpha_1,\ldots,\alpha_n) \in \mathcal{S}_n$ denote the vector of critical values of the step-up MTP $T$. Guo and Rao \cite{Guo} showed the following. 

\begin{lemma}
Let $T$ be any step-up MTP having vector of critical values $\mathbf{\alpha}_T \in \mathcal{S}_n$. The following inequality holds under the PRDS property:
\begin{equation} \label{eq2}
\sum_{k=1}^n \mathbb{P}\left(R_{n}(T)=k \mid P_i \leq \alpha_k\right) \leq 1, \quad \text { for } i \in \mathcal{A} .
\end{equation}
Moreover, the above inequality becomes an equality under the independence of the test statistics.
\end{lemma}
They also constructed an example of the joint distribution of the $p$-values, under which the PRDS property fails to hold although the inequality \eqref{eq2} holds. Thus, it turns out that the inequality \eqref{eq2} is a  strictly weaker property of the test statistics than the PRDS property. They further showed the following optimality property of the BH procedure:

\begin{theorem}\label{BHoptimal}
Let $\mathcal{T}$ be the class of all step-up procedures with vector of cutoffs belonging to $\mathcal{S}_n$ and satisfying the inequality \eqref{eq2}. Then, the Benjamini–Hochberg
procedure is optimal in the class $\mathcal{T}$. That is, for any step-up procedure $T \in \mathcal{T}$ with vector of critical values $\mathbf{\alpha}_T \in \mathcal{S}_n$, if it can control the FDR at $\alpha$, then $\alpha_k \leq k\alpha/n$ for each $k \in \{1, \ldots, n\}$.
\end{theorem}

Theorem \ref{thmBH} and Theorem \ref{BHoptimal} result in the following:
\begin{theorem}\label{BHoptimal2}
Let $\mathcal{T}$ be the class of all step-up procedures with vector of cutoffs belonging to $\mathcal{S}_n$ and satisfying the inequality \eqref{eq2}. Let $T \in \mathcal{T}$ be such that it controls the FDR at $\alpha \in (0,1)$. 
 Consider the equicorrelated normal setup with correlation $\rho$. Then, under the global null, for all $\alpha \in (0,1)$ and for all $\rho \in (0,1)$, 
$$\lim_{n \to \infty} FDR_{T}(n, \alpha, \rho) \leq 1 - \Phi\left[\inf_{t \in  (0,1)} \frac{\Phi^{-1}\left(1-t \alpha\right)-\sqrt{1-\rho} \cdot \Phi^{-1}(1-t)}{\sqrt{\rho}}\right].$$
\end{theorem}

\section{Hommel's Procedure} We have focused on step-down and step-up procedures so far. However, many powerful MTPs proposed in the literature do not belong to the step-down or step-up categories. The Hommel \cite{Hommel} procedure is such a $p$-value based MTP that controls the FWER. The decisions for the individual hypotheses are performed in the following simple way: 

Step 1. Compute $j=\max \left\{i \in\{1, \ldots, n\}: P_{(n-i+k)}>k \alpha / i\right.$ for $\left.k=1, \ldots, i\right\}$. 

Step 2. If the maximum does not exist in Step 1, reject all the hypotheses. 

\hspace{13mm} Otherwise,  reject all $H_i$ with $P_i \leqslant \alpha / j$. 

\noindent Hommel's MTP is uniformly more powerful than the methods of Bonferroni, Holm, and Hochberg \cite{Gou}. The following two results depict the asymptotic behavior of the FWER of Hommel's procedure under the independent normal setup and under the positively equicorrelated normal setup, respectively.

\begin{theorem}\label{thm5.1}
Consider the multiple testing problem under the independent normal setup. Under the global null, we have
$$\lim_{n \to \infty} FWER_{Hommel}(n, \alpha, 0) = 1 - e^{-\alpha}.$$
\end{theorem}

\begin{theorem}\label{thm5.2}
Consider the multiple testing problem under the equicorrelated normal framework with correlation $\rho \in (0,1)$. Then, for any $\alpha \in (0,1)$,
$$\lim_{n \to \infty} FWER_{Hommel}(n, \alpha, \rho) = 0 $$
with probability one under the global null hypothesis.
\end{theorem}

\begin{proof}[\textbf{\upshape Proof of Theorem \ref{thm5.1}.}]
For $1 \leq i \leq n$, we have $P_{(i)}=1-\Phi\left(X_{(n-i+1)}\right)$. Putting $i=n-j+k$ (here $1 \leq j \leq n$ and $1 \leq k \leq j$) gives
$P_{(n-j+k)}=1-\Phi\left(X_{(j-k+1)}\right) \quad k \leqslant j$. 

\newpage
\noindent Now, 
$$
\begin{aligned}
P_{(n-j+k)}>\frac{k \alpha}{j} 
\iff & 1-\Phi\left(X_{(j-k+1)}\right)>\frac{k \alpha}{j} \\
\iff & \Phi^{-1}\left(1-\frac{k \alpha}{j}\right)>X_{(j-k+1)} \\
\iff & \Phi^{-1}\left(1-\frac{s \alpha}{t}\right)>X_{(n(t-s)+1)} \quad \text{where} \hspace{1mm} s=k/n \hspace{1mm}\text{and} \hspace{1mm} t=j/n .
\end{aligned}
$$
\noindent For any $r \in(0,1)$, $X_{(nr)}$ converges in probability to $r$'th quantile of the distribution of $X_1$ as $n \rightarrow \infty$. This implies, $X_{(n(t-s)+1)}$ converges in probability to $\Phi^{-1}(t-s)$ as $n \rightarrow \infty$. Thus, as $n \rightarrow \infty$, 
$$
\begin{aligned}
P_{(n-j+k)}>\frac{k \alpha}{j}
\iff & \Phi^{-1}\left(1-\frac{s \alpha}{t}\right)>\Phi^{-1}(t-s) \\
\iff & 1-\frac{s \alpha}{t}>t-s \\
\iff & t-s \alpha>t(t-s) \\
\iff & t(1-t)>s(\alpha-t).
\end{aligned}
$$

We have $t \geq s$ and $1>\alpha$. So, $t(1-t)>s(\alpha-t)$ always holds. This means that the largest $t$ for which $t(1-t)>s(\alpha-t)$ holds for each $s \in (0,t]$ is 1. This in turn implies that, as $n \to \infty$, the largest integer $j \leq n$ satisfying $P_{(n-j+k)}>\frac{k \alpha}{j}$ for all $k \in \{1, \ldots, j\}$ is $n$ with probability one. Thus, the Hommel's procedure is same as the Bonferroni's procedure as $n \rightarrow \infty$. Hence,
$$\lim_{n \to \infty} FWER_{Hommel} (n, \alpha, 0)=1-e^{-\alpha}.$$
\end{proof}

\begin{proof}[\textbf{\upshape Proof of Theorem \ref{thm5.2}.}]
For the equicorrelated normal framework with correlation $\rho \in (0,1)$, for each $i \geq 1$, we have $X_{i}=U+Z_{i}$. Here $U \sim N(0,\rho)$ is independent of $\left\{Z_{n}\right\}_{n \geq 1}$ and $Z_{i}$’s are i.i.d $N(0,1-\rho)$. 

\vspace{3mm}
We establish Theorem \ref{thm5.2} in the following steps:

\begin{enumerate}
    \item Showing that as $n \rightarrow \infty$,
$$P_{(n-j+k)}>\frac{k \alpha}{j} \hspace{2mm} \text { for all } k=1, \ldots, j 
\iff U < \min _{0<s<t} f(s)
$$where $f(s)=\Phi^{-1}\left(1-s\alpha/t\right) - \sqrt{1-\rho}\cdot \Phi^{-1}(t-s)$.
\item Showing that $$\Phi\left(\frac{-U-\Phi^{-1}(\alpha)}{\sqrt{1-p}}\right)>t \hspace{2mm} \text{implies} \hspace{2mm}U<\min_{0<s<t} f(s).$$
\item Showing that, for each positive integer $m$,
$$ FWER_{Hommel}(n, \alpha, \rho) \leq \mathbb{P}\left[P_{(1)} \leqslant \frac{1}{t_0} \cdot \frac{\alpha}{n}\right] + \mathbb{P}(U \geq m)$$
where $t_{0} = \max _t\left\{t \in(0,1): \min _{0<s<t} f(s)>U\right\}$.
\end{enumerate}

\vspace{3mm}
\noindent We explicate the steps now.\\


\noindent Similar to the previous proof, we have
$$
\begin{aligned}
P_{(n-j+k)}>\frac{k \alpha}{j} 
\iff & 1-\Phi\left(X_{(j-k+1)}\right)>\frac{k \alpha}{j} \\
\iff & \Phi^{-1}\left(1-\frac{k \alpha}{j}\right)>X_{(j-k+1)} \\
\iff & \Phi^{-1}\left(1-\frac{k \alpha}{j}\right)>U+Z_{(j-k+1)} \\
\iff & \Phi^{-1}\left(1-\frac{s \alpha}{t}\right)>U+Z_{(n(t-s)+1)} \quad \text{where} \hspace{1mm} s=k/n \hspace{1mm}\text{and} \hspace{1mm} t=j/n .
\end{aligned}
$$

\noindent For any $r \in(0,1)$, $Z_{(nr)}$ converges in probability to $r$'th quantile of the distribution of $Z_1$ as $n \rightarrow \infty$. This implies, $Z_{(n(t-s)+1)}$ converges in probability to $\sqrt{1-\rho}\cdot \Phi^{-1}(t-s)$ as $n \rightarrow \infty$. Thus, as $n \rightarrow \infty$, 

$$
\begin{aligned}
P_{(n-j+k)}>\frac{k \alpha}{j} 
\iff & U < \Phi^{-1}\left(1-\frac{s \alpha}{t}\right) - \sqrt{1-\rho}\cdot \Phi^{-1}(t-s).
\end{aligned}
$$

\noindent This means, as $n \rightarrow \infty$,
\begin{equation}\label{step1}P_{(n-j+k)}>\frac{k \alpha}{j} \text { for all } k=1, \ldots, j 
\iff U < \min _{0<s<t} f(s)\end{equation}
completing the proof of step 1.

\noindent Now, $t>t-s$ as $s >0$. This implies $\Phi^{-1}(t)>\Phi^{-1}(t-s)$. Consequently, for each $s >0$, $f(s) > g(s)$ where
$g(s)= \Phi^{-1}\left(1-\frac{s\alpha}{t}\right)-\Phi^{-1}(t)$. Thus, 
$$g(s) > U \implies f(s) > U.$$
\noindent Now, 

$$
\begin{aligned}
g(s)>U
\iff & \Phi^{-1}\left(1-\frac{s \alpha}{t}\right)-\sqrt{1-\rho}\cdot \Phi^{-1}(t)>U \\
\iff & \Phi^{-1}\left(1-\frac{s \alpha}{t}\right)>U+\sqrt{1-\rho}\cdot \Phi^{-1}(t) \\
\iff & 1-\frac{s \alpha}{t}>\Phi\left(U+\sqrt{1-\rho}\cdot \Phi^{-1}(t)\right) \\
\iff & \frac{\Phi\left(-U-\sqrt{1-\rho}\cdot \Phi^{-1}(t)\right)}{\alpha}>\frac{s}{t} .
\end{aligned}
$$
\noindent Therefore, if $\frac{\Phi\left(-U-\sqrt{1-\rho}\cdot \Phi^{-1}(t)\right)}{\alpha}>1$ then $\forall s \in(0, t)$, $g(s)>U$. Hence, $\frac{\Phi\left(-U-\sqrt{1-\rho}\cdot \Phi^{-1}(t)\right)}{\alpha}>1$ implies $f(s)>U$ for all $s \in(0, t)$. Now, 
$$
\begin{aligned}
\frac{\Phi\left(-U-\sqrt{1-\rho} \cdot \Phi^{-1}(t)\right)}{\alpha}>1 
\iff & -U-\sqrt{1-\rho} \cdot \Phi^{-1}(t)>\Phi^{-1}(\alpha) \\
\iff & \Phi\left(\frac{-U-\Phi^{-1}(\alpha)}{\sqrt{1-\rho}}\right)>t.
\end{aligned}
$$
\newpage
\noindent Therefore, we have established the following:

\begin{equation}\label{step2}\Phi\left(\frac{-U-\Phi^{-1}(\alpha)}{\sqrt{1-p}}\right)>t \hspace{2mm} \text{implies} \hspace{2mm}U<\min_{0<s<t} f(s),\end{equation}
completing step 2.

\noindent Thus, 
\begin{align*}
t_0:= \max _t\left\{t \in(0,1): \min _{0<s<t} f(s)\geqslant U\right\}
\geqslant & \max _t\left\{t \in(0,1): \Phi\left(\frac{-U-\Phi^{-1}(\alpha)}{\sqrt{1-\rho}}\right)>t\right\}.
\end{align*}
Now, $U<r$ implies $t_0 \geqslant \varepsilon_r$ where
$$
\varepsilon_r=\Phi\left(\frac{-r-\Phi^{-1}(\alpha)}{\sqrt{1-p}}\right).
$$

\noindent So, for every $m \in \mathbb{N}$, there exists
$\varepsilon_m>0$ such that $t_0>\varepsilon_m$ if $U<m$. In other words, there is $\varepsilon_m$ such that $t_0>\varepsilon_m>0$ with probability at least $\mathbb{P}(U<m)$.
This implies, $t_0$ is bounded away from zero with probability one. Now, let
$$
j_0=\max _{1 \leqslant j \leqslant n}\left\{P_{(n-j+k)}>\frac{k \alpha}{j} \text { for all } k=1, \ldots, j\right\}.
$$
Evidently, $j_0 \geqslant n t_0$. Consequently, under the global null,
$$
\begin{aligned}
FWER_{Hommel}(n, \alpha, \rho)
=& \mathbb{P}\left[\bigcup_{i=1}^n\left\{P_i \leqslant \frac{\alpha}{j_0}\right\}\right]\\
\leqslant &\mathbb{P}\left[\bigcup_{i=1}^n\left\{P_i \leqslant \frac{\alpha}{n t_0}\right\}\right] +\mathbb{P}(U \geqslant m) \\
= &\mathbb{P}\left[P_{(1)} \leqslant \frac{1}{t_0} \cdot \frac{\alpha}{n}\right] + \mathbb{P}(U \geqslant m).
\end{aligned}
$$
This completes the proof of Step 3. Now, $\mathbb{P}(U \geq m) \leq \epsilon$ for all $\epsilon >0$ as $m \to \infty$. We claim now that
$$\mathbb{P}\left[P_{(1)} \leqslant \frac{1}{t_0} \cdot \frac{\alpha}{n}\right] \longrightarrow 0 \quad \text{as} \hspace{2mm} n \to \infty.$$
Its proof is precisely the same as the proof of Theorem 2 of \cite{deybhandari} and we therefore omit it. The rest is obvious.
\end{proof}
\begin{remark}
    Suppose $a >0$. The proof of Theorem 2 of \cite{deybhandari} also culminates in the following:
    $$\mathbb{P}_{M_{n}(\rho)}\left[P_{(1)} \leqslant a \cdot \frac{\alpha}{n}\right] \longrightarrow 0 \quad \text{as} \hspace{2mm} n \to \infty$$
    for each $\rho \in (0,1)$. Then, invoking Slepian's inequality, we have the following:
    Let $\Sigma_n$ be the correlation matrix of $X_1, \ldots, X_n$ having $(i,j)$’th entry $\rho_{ij}$ with $\liminf \rho_{ij}=\delta>0$. Suppose $\mu^{\star} = \sup \mu_{i} < \infty$. Then, for any $\alpha \in (0,1)$,
    $$\mathbb{P}_{\Sigma_{n}}\left[P_{(1)} \leqslant a \cdot \frac{\alpha}{n}\right] \longrightarrow 0 \quad \text{as} \hspace{2mm} n \to \infty.$$
    Note that this is a much stronger result than Corollary \ref{cor1}.
\end{remark}
If one replaces $X_{i}$ by $X_{i}+\mu_{i}$ and $U$ by $U+\mu_{i}$ in the proof of Theorem \ref{thm5.2}, one would obtain the following result:

\begin{theorem}\label{thm5.4}
Consider the equicorrelated normal setup with equicorrelation $\rho\in (0,1)$. Suppose $\sup \mu_{i}$ is finite. Then, for any $\alpha \in (0,1)$, 
$$\lim_{n \to \infty} \mathbb{P}_{M_{n}(\rho)}\bigg(R_{n}(Hommel) \geq 1\bigg) = 0 $$
with probability one under any configuration of true and false null hypotheses.
\end{theorem}

We have following as a corollary which extends Theorem \ref{thm5.2} to any configuration of true and false null hypotheses:

\begin{corollary}\label{cor4}
Consider the equicorrelated normal setup with equicorrelation $\rho\in (0,1)$. Suppose $\sup \mu_{i}$ is finite. Then, for any $\alpha \in (0,1)$, 
$$\lim_{n \to \infty} FWER_{Hommel}(n, \alpha, \rho) = 0$$
with probability one under any configuration of true and false null hypotheses.
\end{corollary}

\section{Power Analysis}
We discuss now the asymptotic powers of stepwise procedures. As with the notions of type I and type II error rates, the concept of power can be extended in various ways when moving from single to multiple hypothesis testing \cite{r4}. One such notion of power is \textit{AnyPwr} \cite{r4}, which is the probability of rejecting at least one false null hypothesis. So, for a MTP $T$,
$$AnyPwr_{T} = \mathbb{P}(S_n(T) \geq 1),$$
where $S_{n}(T)$ denotes the number of true positives in MTP $T$. Dey and Bhandari \cite{deybhandari} showed the following regarding the asymptotic power of Bonferroni's method:

\begin{theorem}\label{thm6.1}
Consider the equicorrelated normal setup with equicorrelation $\rho \in (0,1)$. Suppose $\sup \mu_{i}$ is finite. Then, for any $\alpha \in (0,1)$, $AnyPwr_{Bonferroni}$ goes to zero as $n \to \infty$.
\end{theorem}

We mention below some results on stepwise MTPs:

\begin{corollary}\label{cor5}
Let $\Sigma_n$ be the correlation matrix of $X_1, \ldots, X_n$ with $(i,j)$’th entry $\rho_{ij}$ such that $\liminf \rho_{ij}=\delta>0$. Suppose $\sup \mu_{i}$ is finite and $T$ is any step-down MTP  controlling FWER at level $\alpha \in (0,1)$. Then, for any $\alpha \in (0,1)$, $AnyPwr_{T}$ goes to zero as $n \to \infty$.
\end{corollary}

 \begin{corollary}\label{cor6}
Consider the equicorrelated normal setup with equicorrelation $\rho\in (0,1)$. Suppose $\sup \mu_{i}$ is finite. Then, for any $\alpha \in (0,1/2)$, $AnyPwr_{Hochberg}$ goes to zero as $n \to \infty$.
\end{corollary}

\begin{corollary}\label{cor7}
Consider the equicorrelated normal setup with equicorrelation $\rho\in (0,1)$. Suppose $\sup \mu_{i}$ is finite. Then, for any $\alpha \in (0,1)$, $AnyPwr_{Hommel}$ goes to zero with probability one as $n \to \infty$ .
\end{corollary}

Corollary \ref{cor5}, \ref{cor6}, \ref{cor7} follow from Corollary \ref{cor2}, \ref{cor3} and Theorem \ref{thm5.4}, respectively, since $S_{n}(T) \leq R_n(T)$ for any MTP $T$.

\section{Concluding Remarks}
In recent years, substantial efforts have been made to understand the properties of multiple testing procedures under dependence. The article \cite{deybhandari} sheds light on the extent of the conservativeness of the Bonferroni method under dependent setups. However, there is little literature on the effect of correlation on general step-down or step-up procedures. This paper addresses this gap in a unified manner by investigating the limiting behaviors of several testing rules under the correlated Gaussian sequence model. We have proved asymptotic zero results for some popular MTPs controlling FWER at a pre-specified level. Specifically, we have shown that the limiting FWER approaches zero for any step-down rule provided the infimum of the correlations is strictly positive. 

The authors in \cite{HuangHsu} show that both Holm's and Hochberg's methods are special cases of partition testing. Holm's MTP tests each partition hypothesis using the maximum order statistic, setting a cutoff utilizing the Bonferroni inequality. Hochberg's procedure, on the contrary, tests each partition hypothesis using each of the order statistics, using a set of cutoffs utilizing Simes’ inequality. It is natural to expect partition testing utilizing the joint distribution of the test statistics is sharper than partition testing based on probabilistic inequalities. Our results elucidate that, at least under the correlated Gaussian sequence model setup with many hypotheses, Holm's MTP and Hochberg's MTP do not have significantly different performances in that they both have asymptotic zero FWER and asymptotic zero power.

The Benjamini-Hochberg procedure has been one of the most studied MTP and has several desirable optimality properties \cite{Bogdan}, \cite{Guo}. It is astonishing to note that, among all the procedures studied in this paper, the BH method is the only one which can hold the FWER at a strictly positive level asymptotically under the equicorrelated normal setup. An interesting problem would be to study the limiting power of the Benjamini-Hochberg method.

Hommel's method is more rejective than Hochberg's MTP (and consequently, Holm's and Bonferroni's methods) \cite{Gou}. Yet, within our chosen asymptotic framework, this has asymptotic zero FWER and asymptotic zero power.

Finally, there are possible scopes of interesting extensions in several directions. One extension is to consider more general distributional setups. Another is to study the limiting behaviors of Hochberg, Hommel, and Benjamini-Hochberg procedures under general dependent normality. The primary tool in establishing universal asymptotic zero results for the step-down MTPs is Slepian's inequality which compares the quadrant probabilities of two normal random vectors. However, for the step-up procedures, the FWERs become functions of several order statistics. Hence we can not directly apply Slepian's inequality in these scenarios. Indeed, the authors in \cite{FDR2007} remark that it is challenging to deal with false discoveries in models with complicated dependence structures, e.g., in a multivariate Gaussian model with a general covariance matrix. It is also interesting to theoretically investigate whether similar asymptotic results hold for other classes of MTPs, e.g., the class of consonant procedures \cite{Westfall}. 


\section*{Acknowledgements}

The author sincerely acknowledges Prof. Subir Kumar Bhandari for his valuable and constructive suggestions during the planning and development of this work.


\vspace{4mm}

\noindent \textbf{Declarations of interest:} none.

\end{document}